\documentclass[10pt]{article}
\usepackage[T1]{fontenc}
\usepackage[utf8]{inputenc}
\usepackage{lmodern}
\usepackage{amsmath,amssymb,amsthm,mathtools}
\usepackage[margin=1in]{geometry}
\usepackage{microtype}
\usepackage{enumitem}
\usepackage{indentfirst}
\usepackage[colorlinks=true,linkcolor=blue,citecolor=blue,urlcolor=blue]{hyperref}
\hypersetup{
	pdftitle={A Note on Long Cycles in 2-Connected Tough Graphs},
	pdfauthor={Songling Shan},
	pdfkeywords={Toughness, Circumference, Treedepth, Bounded-degree spanning tree}
}

\newtheorem{theorem}{Theorem}[section]
\newtheorem{conjecture}[theorem]{Conjecture}
\newtheorem{lemma}[theorem]{Lemma}

\newtheorem{question}[theorem]{Question}
\theoremstyle{remark}

\DeclareMathOperator{\td}{td}
\DeclareMathOperator{\cir}{cir}

\begin{document}
	
	\title{Long Cycles in 2-Connected Tough Graphs}
	\author{
		Songling Shan\thanks{Auburn University, Department of Mathematics and
			Statistics, Auburn, AL 36849. Email: \texttt{szs0398@auburn.edu}.
			Partially supported by NSF grant DMS-2451895.}
	}
	\date{\today}
	\maketitle
	
	\begin{abstract}
		Let $G$ be a graph. The circumference of $G$, denoted by $\cir(G)$, is
		the length of a longest cycle in $G$, or zero if $G$ is acyclic. In
		1993, Broersma, van den Heuvel, Jung, and Veldman conjectured that,
		for every $t>0$, there is a constant $A=A(t)>0$ such that every
		2-connected $t$-tough graph of order $n$ has circumference at least
		$A\log n$; the conjecture is recorded as Conjecture~2 in the 2006
		survey on toughness by Bauer, Broersma, and Schmeichel. In this note,
		we confirm the conjecture. More precisely, every 2-connected $t$-tough
		graph $G$ of order $n$ satisfies
		$\cir(G)\ge \lceil \log_k((k-1)n+1)\rceil$, where
		$k=\lceil 1/t\rceil+2$. The proof combines Win's bounded-degree
		spanning tree theorem with the theorem of Bria\'nski, Joret, Majewski,
		Micek, Seweryn, and Sharma that the treedepth of a 2-connected graph
		is at most its circumference.
	\end{abstract}

	
	\section{Introduction}
	
	We consider only finite simple graphs. Let $G$ be a graph. Denote by $V(G)$
	and $E(G)$ the vertex set and the edge set of $G$, respectively, and let
	$c(G)$ be the number of components of $G$. For $v\in V(G)$, let $d_G(v)$
	be the degree of $v$, and let $\Delta(G)$ be the maximum degree of $G$.
	For $S\subseteq V(G)$, let $G-S$ be the graph obtained from $G$ by
	deleting the vertices of $S$; we write $G-v$ for $G-\{v\}$. The
	\emph{circumference} of $G$, denoted by $\cir(G)$, is the length of a
	longest cycle in $G$, or zero if $G$ is acyclic. All logarithms without a
	subscript are to base $2$.
	
	Let $t>0$. A graph $G$ is \emph{$t$-tough} if $|S|\ge t\,c(G-S)$ for every
	$S\subseteq V(G)$ with $c(G-S)\ge 2$. The \emph{toughness} $\tau(G)$ of a
	noncomplete graph $G$ is the largest real number $t$ for which $G$ is
	$t$-tough; as usual, $\tau(K_n)=\infty$. Note that a $t$-tough graph is
	$t'$-tough for every $0<t'\le t$.
	
	For a positive integer $r$, let $\gamma_r(t,n)$ denote the minimum
	circumference among all $r$-connected $t$-tough graphs of order $n$.
	Broersma, van den Heuvel, Jung, and Veldman~\cite{BroersmaEtAl} proved
	that, for fixed $t>0$,
	$\gamma_2(t,n)\log \gamma_2(t,n)\ge(2-o(1))\log n$, and they constructed
	examples showing that $\gamma_2(t,n)=O(\log n)$ when $0<t\le1$. They
	conjectured that the logarithmic bound is the correct order of magnitude.
	The conjecture also appears as Conjecture~2 in the survey of Bauer,
	Broersma, and Schmeichel~\cite{BauerBroersmaSchmeichel}.
	
	\begin{conjecture}[{Broersma, van den Heuvel, Jung, and Veldman~\cite[Conjecture 19]{BroersmaEtAl}}]\label{conj:survey}
		For every $t>0$, there is a positive constant $A=A(t)$ such that
		$\gamma_2(t,n)\ge A\log n$.
	\end{conjecture}
	
	Broersma, van den Heuvel, Jung, and Veldman~\cite{BroersmaEtAl} verified Conjecture~\ref{conj:survey}
	for 3-connected graphs, and for 2-connected $K_{1,\ell}$-free graphs
	with $\ell$ fixed (in particular, for 2-connected graphs of bounded
	maximum degree). Our main result confirms the conjecture in general and
	gives an explicit bound.
	
	\begin{theorem}\label{thm:main}
		Let $t>0$, and let $G$ be a 2-connected $t$-tough graph of order $n$. Set
		$k=\lceil 1/t\rceil+2$. Then
		\[
		\cir(G)\ge \left\lceil \log_k\bigl((k-1)n+1\bigr)\right\rceil
		\ge \frac{\log n}{\log k}.
		\]
	\end{theorem}
	
	The examples of Broersma, van den Heuvel, Jung, and Veldman~\cite{BroersmaEtAl} show that the
	logarithmic order in Theorem~\ref{thm:main} cannot be improved when
	$0<t\le1$. On the other hand, $k=3$ for every $t\ge1$, so the bound of
	Theorem~\ref{thm:main} does not improve as $t$ grows beyond $1$, whereas
	$\gamma_2(t,n)$ itself is expected to: Chv\'atal~\cite{Chvatal}
	conjectured that there is a constant $t_0$ such that every $t_0$-tough
	graph on at least three vertices is Hamiltonian, which would give
	$\gamma_2(t,n)=n$ for all $t\ge t_0$. Bauer, Broersma, and
	Veldman~\cite{BauerBroersmaVeldman} showed that necessarily
	$t_0\ge\frac94$.
	
	The survey~\cite{BauerBroersmaSchmeichel} also asks whether
	3-connectivity forces a polynomially long cycle in tough graphs, that is,
	whether $\gamma_3(t,n)\ge n^{B}$ for some positive constant $B=B(t)$.
	Broersma, van den Heuvel, Jung, and Veldman~\cite[Theorem~7]{BroersmaEtAl} constructed 3-connected
	1-tough graphs of order $n$ with circumference $O(\log n)$; since a
	1-tough graph is $t$-tough for every $0<t\le1$, this shows that
	$\gamma_3(t,n)=O(\log n)$ for all $0<t\le1$, and so the question is of
	interest only for $t>1$. We record it in this form.
	
	\begin{question}\label{ques:three}
		Let $t>1$. Is there a positive constant $B=B(t)$ such that
		$\gamma_3(t,n)\ge n^{B}$ for all $n$?
	\end{question}
	
	If Chv\'atal's conjecture is true, then Question~\ref{ques:three} has an
	affirmative answer for all $t\ge t_0$, so the question really concerns
	the range $1<t<t_0$.
	
	The proof of Theorem~\ref{thm:main} is short. By a theorem of
	Win~\cite{Win}, every connected $t$-tough graph has a spanning tree of
	maximum degree at most $\lceil 1/t\rceil+2$. Such a tree has treedepth
	logarithmic in its order. Finally, Bria\'nski, Joret, Majewski, Micek,
	Seweryn, and Sharma~\cite{BrianskiEtAl} proved that the treedepth of a
	2-connected graph is at most its circumference. In
	Section~\ref{sec:proof}, we recall these two theorems, establish the
	elementary treedepth bound for bounded-degree trees, and prove
	Theorem~\ref{thm:main}.
	
	\section{Proof of Theorem~\ref{thm:main}}\label{sec:proof}
	
	We first recall Win's bounded-degree spanning tree theorem. A spanning tree
	of maximum degree at most $k$ is sometimes called a \emph{$k$-tree}.
	
	\begin{theorem}[Win~\cite{Win}]\label{thm:win}
		Let $k\ge3$ be an integer, and let $G$ be a connected graph. If
		$c(G-S)\le (k-2)|S|+2$ for every $S\subseteq V(G)$, then $G$ has a
		spanning tree of maximum degree at most $k$.
	\end{theorem}
	
	We next recall the recursive definition of treedepth; see~\cite[Chapter~6]{NesetrilOssona} for background. Set
	$\td(\varnothing)=0$. If $G$ is disconnected, then
	$\td(G)=\max\{\td(D): D\text{ is a component of }G\}$, and if $G$ is
	nonempty and connected, then $\td(G)=1+\min_{v\in V(G)}\td(G-v)$.
	Treedepth is monotone under taking subgraphs. Moreover, a connected graph
	has treedepth $1$ if and only if it has exactly one vertex.
	
	The following theorem is the main result of~\cite{BrianskiEtAl}.
	
	\begin{theorem}[Bria\'nski et al.~\cite{BrianskiEtAl}]\label{thm:tdcirc}
		Let $G$ be a 2-connected graph. Then $\cir(G)\ge \td(G)$.
	\end{theorem}
	
	We also need one elementary observation about bounded-degree trees.
	
	\begin{lemma}\label{lem:tree}
		Let $k\ge2$ and $d\ge1$ be integers, and let $T$ be a tree with
		$\Delta(T)\le k$ and $\td(T)\le d$. Then
		$|V(T)|\le 1+k+k^2+\cdots+k^{d-1}=(k^d-1)/(k-1)$. Consequently, every
		tree $T$ with $\Delta(T)\le k$ satisfies
		$\td(T)\ge \lceil \log_k((k-1)|V(T)|+1)\rceil$.
	\end{lemma}
	
	\begin{proof}
		We prove the first assertion by induction on $d$. If $d=1$, then $T$ is
		a connected graph of treedepth $1$, so $|V(T)|=1$. Let $d\ge2$. If
		$\td(T)\le d-1$, then the assertion follows from the induction
		hypothesis, so we may assume that $\td(T)=d$. Since $T$ is connected,
		the recursive definition of treedepth gives
		$d=\td(T)=1+\min_{v\in V(T)}\td(T-v)$, and hence there is a vertex
		$v\in V(T)$ with $\td(T-v)=d-1$; that is, every component of $T-v$ has
		treedepth at most $d-1$. Since $T$ is a tree, $T-v$ has exactly
		$d_T(v)\le k$ components, each of which is a tree of maximum degree at
		most $k$. By the induction hypothesis, each component of $T-v$ has at
		most $1+k+\cdots+k^{d-2}$ vertices. Therefore,
		$|V(T)|\le 1+k\,(1+k+\cdots+k^{d-2})=1+k+\cdots+k^{d-1}$.
		
		For the second assertion, apply the first with $d=\td(T)$: it gives
		$(k-1)|V(T)|+1\le k^{d}$, that is, $d\ge \log_k((k-1)|V(T)|+1)$. Since
		$d$ is an integer, the claimed inequality follows.
	\end{proof}
	
	\begin{proof}[Proof of Theorem~\ref{thm:main}]
		Set $k=\lceil 1/t\rceil+2$. Then $k\ge3$ and $k-2\ge1/t$. Let
		$S\subseteq V(G)$. If $c(G-S)\ge2$, then the $t$-toughness of $G$ gives
		$c(G-S)\le |S|/t\le(k-2)|S|$, and if $c(G-S)\le1$, then trivially
		$c(G-S)\le(k-2)|S|+2$. Thus the hypothesis of Theorem~\ref{thm:win}
		holds, and $G$ has a spanning tree $T$ with $\Delta(T)\le k$. By
		Lemma~\ref{lem:tree}, $\td(T)\ge \lceil \log_k((k-1)n+1)\rceil$. Since
		$T$ is a subgraph of $G$, the monotonicity of treedepth yields
		$\td(G)\ge\td(T)$. Theorem~\ref{thm:tdcirc} now gives
		\[
		\cir(G)\ge\td(G)\ge\td(T)
		\ge \left\lceil \log_k\bigl((k-1)n+1\bigr)\right\rceil.
		\]
		Finally, since $(k-1)n+1\ge n$, we have
		$\lceil \log_k((k-1)n+1)\rceil\ge\log_k n=\log n/\log k$. This
		completes the proof.
	\end{proof}
	
	\section*{Declaration on the Use of AI Tools}
	
	While the author was exploring proof ideas, ChatGPT 5.6 Sol located the
	paper~\cite{BrianskiEtAl}, which immediately led to the short proof
	presented here. Claude Fable 5.1 was used to assist with language
	editing, grammar, clarity, and formatting. The author reviewed and
	verified all AI-assisted edits and takes full responsibility for the
	content of the manuscript.

\end{document}